\documentclass{article}
\title{Explicit reduction modulo $p$ of certain 2-dimensional
crystalline representations.}
\author{Kevin Buzzard, Toby Gee} 

\usepackage{amsmath} 

\usepackage{amssymb}

\usepackage{a4wide}

\usepackage[all]{xy}

\newcommand{\Bbar}{\overline{B}}

\newcommand{\Fps}{\mathbf{F}_{p^2}}
\newcommand{\Fpbar}{\overline{\mathbf{F}}_p}
\newcommand{\GQ}{\Gal(\Qbar/\Q)}
\newcommand{\GQp}{\Gal(\Qpbar/\Qp)}
\newcommand{\rhobar}{\overline{\rho}}
\newcommand{\Q}{\mathbf{Q}}
\newcommand{\Qbar}{\overline{\Q}}
\newcommand{\Qp}{\Q_p}
\newcommand{\Qpbar}{\overline{\Q}_p}
\newcommand{\Thetabar}{\overline{\Theta}}
\newcommand{\thetabar}{\overline{\theta}}
\newcommand{\Vbar}{\overline{V}}
\newcommand{\Z}{\mathbf{Z}}
\newcommand{\Zp}{\Z_p}
\newcommand{\Zpbar}{\overline{\Z}_p}

\DeclareMathOperator{\cris}{cris}
\DeclareMathOperator{\coker}{coker}

\DeclareMathOperator{\End}{End}
\DeclareMathOperator{\Frob}{Frob}
\DeclareMathOperator{\Id}{Id}
\DeclareMathOperator{\ind}{ind}
\DeclareMathOperator{\im}{im}
\DeclareMathOperator{\Gal}{Gal}
\DeclareMathOperator{\GL}{GL}

\DeclareMathOperator{\sss}{ss}
\DeclareMathOperator{\SL}{SL}
\DeclareMathOperator{\St}{St}
\DeclareMathOperator{\Symm}{Symm}

\def\smallmat#1#2#3#4{\bigl(\begin{smallmatrix}{#1}&{#2}\\{#3}&{#4}\end{smallmatrix}\bigr)}

\usepackage{amsthm}
\theoremstyle{plain}
\newtheorem{theorem}{Theorem}[section]
\newtheorem{proposition}[theorem]{Proposition}
\newtheorem{defn}[theorem]{Definition}
\newtheorem{lemma}[theorem]{Lemma}
\newtheorem{corollary}[theorem]{Corollary}
\theoremstyle{remark}
\newtheorem{remark}[theorem]{Remark}
\newtheorem{question}[theorem]{Question}

\begin{document}
\maketitle

\begin{abstract}We use the $p$-adic local Langlands correspondence for
$\GL_{2}(\Qp)$ to explicitly compute the reduction modulo $p$ of certain
2-dimensional crystalline
representations of small slope, and give applications to modular
forms.\end{abstract}

\section{Introduction.}

Let $f=\sum_{n\geq1} a_nq^n$ be a weight~$k$ cusp form for the group
$\Gamma_1(N)\subseteq\SL_2(\Z)$,
and assume that $f$ is normalised ($a_1=1$), is an eigenform
for all the Hecke operators, and has character $\psi$ (a Dirichlet
character modulo~$N$). The coefficients of $f$ are
complex numbers, but are well-known to be algebraic over $\Q$
and hence can be regarded (after some choices) as elements
of $\Qpbar$, where $p$ is a prime number. Deligne associated
to $f$ a $p$-adic Galois representation
$$\rho_f:\GQ\to\GL_2(\Qpbar).$$
Let us normalise the construction so that if $f$ is associated
to an elliptic curve over $\Q$ then $\rho_f$ is the Tate module of
the curve; this choice of normalisation is sometimes not ideal,
but our main results appear cleaner with this choice.

A lot of explicit information is known about
the ``local'' structure of
$\rho_f$, by which we mean $\rho_f|_{D_\ell}$, where $D_\ell$ denotes
a decomposition group at some prime number $\ell$.
For example if $\ell$ is a prime not dividing $Np$
then $\rho_f$ is unramified at $\ell$ and the characteristic polynomial
of $\rho_f(\Frob_\ell)$ (with $\Frob_\ell$ an arithmetic Frobenius)
is $X^2-a_\ell X+\ell^{k-1}\psi(\ell)$. In particular the local structure
of $\rho$ at $\ell$ is determined by the local structure of $f$ at
$\ell$, which, because $\ell\nmid N$, can be interpreted as
the triple $(a_\ell, k,\psi(\ell))$.
This is visibly an explicit description of $\rho_f|_{D_\ell}$ (there is
still a question regarding semisimplicity of $\rho_f(\Frob_\ell)$ in the
case where the two eigenvalues coincide, but conjecturally this should
never occur in weight $k\geq2$, and in weight~1 $\rho_f(\Frob_\ell)$
is known to be semisimple anyway).

If $\ell|N$ but $\ell\not=p$
then the local Langlands conjecture for $\GL_2$ (a bijection which
can be explicitly written down in essentially all cases) and a ``local-global''
theorem
of Carayol (following Deligne and Langlands) again gives us an explicit
description of $\rho_f|_{D_\ell}$.

Let us now turn our attention to the local structure of $\rho_f$
at $p$. Now one might argue that the theorems
describing $\rho_f|_{D_p}$ in terms of the data attached to $f$
are far from ``explicit''---and indeed, how could one expect
them to be explicit: a 2-dimensional $p$-adic representation
of $D_p$ is a very complicated object. However the \emph{mod $p$}
2-dimensional representations of $D_p$ are easily classified:
the reducible ones are, up to semisimplification, the sum of two
characters, and the irreducible ones are all induced from characters
of the absolute Galois group of the unramified quadratic
extension of $\Q_p$. So using a little local class field theory
it is easy to explicitly list all these representations.
The following ``practical'' question
then arises:

\begin{question} If $f=\sum a_nq^n$ is a normalised cuspidal
level~$N$ eigenform and $p$ is a prime, and if $\overline{\rho}_f$
is the associated semisimple representation
$$\overline{\rho}_f:\GQ\to\GL_2(\Fpbar),$$
then can one explicitly read off $(\overline{\rho}_f|_{D_p})^{\sss}$
from the weight, character and $q$-expansion of $f$?
\end{question}

In this generality, the answer to the question is in some
sense ``no''. For example, if $f$ is a newform of level $\Gamma_0(N)$
and $p$ divides $N$ exactly once, then to know $\rho_f|_{D_p}$
is (essentially) to know the value of the so-called $\cal{L}$-invariant
of $f$ at $p$, and
this invariant is subtle: as far as anyone knows, it cannot be easily read
off from the $q$-expansion of $f$, and is not a ``local''
invariant of the classical automorphic representation attached to $f$.
Furthermore, the problem does not go away when reducing mod $p$:
$\overline{\rho}_f|_{D_p}$ also depends heavily on the $\cal{L}$-invariant,
even for small weight modular forms: see for example Th\'eor\`eme~4.2.4.7
of~\cite{breuil-mezard} for some examples of how the $\cal{L}$-invariant
affects the local mod~$p$ representation.
If furthermore $p$ divides $N$ more than once,
then even less is known, and the explicit dictionary is no doubt
even more complicated.

However, if $p\nmid N$ the situation is much better. The local data
attached to $f$ at $p$ is the triple $(a_p,k,\psi(p))$ (in the sense
that the local component $\pi_p$ of the automorphic representation
$\pi$ associated to $f$ is completely determined by this data), and in some
cases this local data does determine a lot about the local representation.
We now explain some of the results known in this situation. First
we introduce some notation.

We identify $D_p$ with the local Galois group $\GQp$.
Let $\mu_\alpha$ denote the unramified character of $D_p$ sending
a geometric Frobenius to $\alpha$. Let us normalise the isomorphisms
of local class field theory by identifying uniformisers with geometric
Frobenii. Let  $\epsilon:D_p\to\Z_p^\times$ denote
the cyclotomic character, and let $\omega:D_p\to\Fpbar^\times$ denote
the mod~$p$ reduction of~$\epsilon$.
Let $\omega_2:\Gal(\Qpbar/\Q_{p^2})\to\Fpbar^\times$
denote a character such that the induced map $\Q_{p^2}^\times\to\Fpbar^\times$
sends~$p$ to~$1$ and such that the induced map
$\Fps^\times\to\Fpbar^\times$ is induced by a morphism
of fields $\Fps\to\Fpbar$. There are two such choices for $\omega_2$;
we fix one.
We call $\omega_2$ a ``fundamental character of niveau 2''.
The other fundamental character is $\omega_2^p$. 
Note that $\omega_2^{p+1}=\omega$ on $I_p$, the inertia
subgroup of $D_p$. Abusing notation, let us also use $\omega_2$ to mean
$\omega_2|_{I_p}$. By local class field theory we can also consider
$\omega$ and $\omega_2$ as characters of $\Q_p^\times$ and $\Q_{p^2}^\times$
respectively.

The following theorems were proved by global methods:

\begin{theorem}\label{thm:deligneserrefontaineedixhoven}

(1) (Deligne-Serre) If $k=1$ then $\rho_f$ is unramified at $p$,
and $\rho_f(\Frob_p)$ is semisimple with characteristic
polynomial $X^2-a_pX+\psi(p)$.

(2) (Deligne) If $k\geq2$
and $a_p$ is a $p$-adic unit then $\rho_f|_{D_p}$ is reducible
(and may or may not be semisimple), and the semisimplification
of $\rho_f|_{D_p}$ is isomorphic to
$\mu_{a_p^{-1}}\oplus\epsilon^{k-1}\mu_{a_p/\psi(p)}$.

(3) (Fontaine, Edixhoven) If $2\leq k\leq p+1$ and $a_p$
is not a $p$-adic unit then $\rhobar_f|_{D_p}$ is irreducible,
and $\rhobar_f|_{I_p}\cong\omega_2^{k-1}\oplus\omega_2^{p(k-1)}$.

\end{theorem}
\begin{remark} Part (1) of the theorem was proved by Deligne and Serre
in~\cite{deligne-serre}. 
Part (2) was proved in a~1974 letter from Deligne to Serre which
apparently has never been published, although published proofs
are now in the literature (see for example Theorem~2
of~\cite{wiles-inv}). Part (3) was proved (for $k\leq p$) in 1979
by Fontaine in two letters to Serre, but as far as we know the
first published proof was given by Edixhoven in~\cite{edix:weights}
and this proof uses global methods. Note finally the relative
strengths of the results: (1) and (2) describe the $p$-adic
representation, and (2) has no restrictions on the weight, whereas
(3) only describes the mod~$p$ representation and only for
small weights---this is because parts (1) and (2) concern
ordinary modular forms, and the situation is much more complicated
in the non-ordinary case.
\end{remark}

The problem with extending these global methods to the higher weight
non-ordinary case is that they typically rely on the arithmetic
of mod $p$ modular forms
and the geometry of modular curves over $\Fpbar$, and hence find
it very hard to distinguish between $a_p$s with positive
valuations. However for $k>p+1$ one can easily find examples
on a computer of modular forms $f_1$ and $f_2$ of the same
weight and level (prime to $p$) and character,
with $v(a_p(f_1))>0$, $v(a_p(f_2))>0$ and
$(\rhobar_{f_1}|_{I_p})^{\sss}\not\cong(\rhobar_{f_2}|_{I_p})^{\sss}$. In particular,
vaguely speaking, the mod $p$ Galois representation associated
to a modular form, locally at $p$,
depends on more than the mod $p$ reduction of the local data
attached to the form at $p$.

There is however a completely different and far more local
way of approaching the problem, which relies on a
coincidence in $p$-adic Hodge
Theory which is very specific to 2-dimensional representations
of $\GQp$. 
Let $f$ be a normalised cuspidal eigenform of level~$N$
prime to~$p$. Then $\rho_f|_{D_p}$ is a crystalline Galois representation,
and we would like to say something concrete about $\rhobar_f|_{D_p}$.
Because of Theorem~\ref{thm:deligneserrefontaineedixhoven} above,
we may now restrict our attention
to the case $k\geq2$ and $v(a_p)>0$. Moreover, after an
unramified twist we may assume $\psi(p)=1$.
 The coincidence is that if we furthermore assume that the roots
of $X^2-a_pX+p^{k-1}$ are distinct (or equivalently that $a_p^2\not=4p^{k-1}$,
an inequality which always holds if $k=2$ and conjecturally holds
for any $k\geq2$) then there is up to isomorphism a \emph{unique}
$\Qpbar$-vector space $D_{k,a_p}$
equipped with a semisimple linear Frobenius $\phi$ with
characteristic polynomial $X^2-a_pX+p^{k-1}$ and a weakly
admissible filtration with jumps at $0$ and $k-1$.
In this case we have the following
theorem which follows from Scholl's work on modular forms and
the $p$-adic comparison isomorphism:

\begin{theorem}\label{1.4} (Scholl, Faltings)

If $k\geq2$, if $v(a_p)>0$, and furthermore if
$a_p^2\not=4p^{k-1}$, then $D_{\cris}((\rho_f|_{D_p})^*)\cong D_{k,a_p}$.

\end{theorem}
Note that the omitted case $a_p^2=4p^{k-1}$
does not occur if $k=2$ and should not occur if $k>2$ (it would contradict
a conjecture of Tate: see~\cite{coleman-edixhoven}). Note also that
we need to take the $\Qpbar$-dual of $\rho_f$; this is because of
our conventions.

Again under the assumptions $k\geq2$,
$a_p\in\Qpbar$ with $|a_p|<1$ and $a_p^2\not=4p^{k-1}$,
let $V_{k,a_p}$ denote the crystalline representation $V$
such that $D_{\cris}(V_{k,a_p}^*)\cong D_{k,a_p}$.
Our conclusion is that (with notation as above)
$\rho_f|_{D_p}\cong V_{k,a_p}$, and in particular
the mod $p$ reduction of $\rho_f|_{D_p}$ is determined by the local data
$(a_p,k)$ associated to $f$. But this reconstruction
is \emph{far} from explicit! It leads us to formulate the explicit
purely local

\begin{question} Say $k\geq2$ and $a_p\in\Qpbar$ with $v(a_p)>0$
and $a_p^2\not=4p^{k-1}$. What is the isomorphism
class of $\Vbar_{k,a_p}$ (the semisimplification of the
reduction of the $D_p$-representation $V_{k,a_p}$)
as an explicit function of $k$ and $a_p$?
\end{question}

A few years ago this question seemed to be regarded as almost intractible
for weights $k>p$: as far as we know, the only  results
for high weights were those of Berger, Li and Zhu (\cite{berger-li-zhu})
who showed that for $v(a_p)$ sufficiently large (an explicit bound
depending on $k$) the answer was the same as for the case $a_p=0$
(which was already known). Not only that, computational evidence
collected by one of us (KB) seemed to indicate that the answer
to the question was in general rather subtle.

However, recent work of Breuil, Berger
and Colmez on the $p$-adic and
mod~$p$ local Langlands correspondence for $\GL_2(\Q_p)$ gives us
a completely new approach for attacking this problem. We now summarise
what we need to know about the $p$-adic and mod $p$ Langlands
correspondence for $\GL_2(\Q_p)$ here (although we do not go into the details
of extension classes, an important subtlety which we will not need
here, and we shall only discuss the case
of crystalline $p$-adic representations; much
is now known in more general cases but we shall not need these results). 
Note also that from this point on, our notational conventions will
be that a bar over a $\Qpbar$-vector space (a Galois representation,
or a representation of $\GL_2(\Q_p)$) will mean
(some sensible variant of) ``take a stable lattice, reduce, and then
semisimplify'', whereas a bar over a ``lattice'' (for example
the $\Zpbar$-module $\Theta_{k,a_p}$ that we shall see later on) will
mean ``reduce mod~$p$ and don't semisimplify''.

Breuil in~\cite{breuil1}
has classified the irreducible smooth admissible representations
of $\GL_2(\Q_p)$ over $\Fpbar$. Furthermore he also wrote down an explicit
injective map $\Vbar\mapsto LL(\Vbar)$ from the set of isomorphism
classes of semisimple 2-dimensional $\Fpbar$-representations of $\GQp$
to the set of isomorphism classes of semisimple finite length admissible
smooth representations of $\GL_2(\Q_p)$, the idea being that this
map should be the ``correct'' version of the Local Langlands correspondence
in this setting. Note that the image
of the map $LL$ does not contain every irreducible representation
of $\GL_2(\Q_p)$ and it does contain some reducible ones. In particular
the situation is not quite as simple as the classical local Langlands
correspondence.

Now let $V$ denote an irreducible 2-dimensional crystalline representation
of $\GQp$ over~$\Qpbar$.
Berger and Breuil in~\cite{berger-breuil} associate to~$V$
a $p$-adic Banach space $B(V)$ equipped with a unitary action of $\GL_2(\Q_p)$,
such that $B(V)$ is topologically irreducible
(see Corollaire~5.3.2 of~\cite{berger-breuil}).
The map $B$ should be thought of as a $p$-adic local Langlands correspondence.
If one chooses an open $\GL_2(\Q_p)$-stable lattice in this Banach space
and then tensors
the lattice with $\Fpbar$, one obtains a finite length $\Fpbar$-representation
of $\GL_2(\Q_p)$ whose semisimplification $\Bbar(V)$ depends only on $V$. 
A crucial theorem of Berger (Th\'eor\`eme~A of~\cite{ber05})
is that $\Bbar(V)\cong LL(\Vbar)$,
where $\Vbar$ denotes the semisimplification of the
$\Fpbar$-reduction of $V$. In words,
the $p$-adic and mod $p$ local Langlands dictionaries
are compatible. In particular $\Bbar(V)$ actually only depends
on $\Vbar$.

The final ingredient in our approach is the following.
We have our data $(a_p,k)$ with as usual
$k\geq2$, $v(a_p)>0$ and $a_p^2\not=4p^{k-1}$.
Given this data, Breuil constructs an irreducible
(algebraic) $\Qpbar$-representation $\Pi_{k,a_p}$
of the group $\GL_2(\Q_p)$, and this representation
stabilises a lattice $\Theta_{k,a_p}$, so
$\Thetabar_{k,a_p}:=\Theta_{k,a_p}\otimes\Fpbar$
is an $\Fpbar$-representation of $\GL_2(\Qp)$ (we will see explicit
definitions later; the constructions were made in~\cite{breuil2}
and the proofs of the assertions we are making here are in loc.\ cit.\ and
in section~5 of~\cite{berger-breuil}).
Breuil and Berger have shown (again in section~5 of~\cite{berger-breuil})
that $B(V_{k,a_p})$ is
a certain completion of $\Pi_{k,a_p}$ and it follows that
$\Thetabar_{k,a_p}^{\sss}\cong \Bbar(V_{k,a_p})$. In particular,
$\Thetabar_{k,a_p}^{\sss}$ \emph{is in the image of $LL$}.
This opens up the possibility of computing it via a process of
elimination: if we examine the image of $LL$
and manage to rule out all but one element of it as a possibility
for $\Thetabar_{k,a_p}^{\sss}$ then we have computed
$\Thetabar_{k,a_p}^{\sss}$ and hence $\Vbar_{k,a_p}$!
Note that this approach
relies heavily on both Breuil's classification and the explicit
happy $(\phi,\Gamma)$-module coincidences for $\GL_2(\Q_p)$
implicit in Berger's work, and hence seems to be restricted
to the case of 2-dimensional representations of $\GQp$,
but it does enable us to give the first explicit computations
of $\Vbar_{k,a_p}$ valid in the non-ordinary
case where $v(a_p)$ is small and $k$ is unbounded. We remark
that Paskunas independently proved similar results using similar
ideas.

Note that if $1\leq t\leq p$ then $\ind(\omega_2^t)$ (induction
from $\Gal(\Qpbar/\Q_{p^2})$ to $\GQp$) is an irreducible
2-dimensional representation of $\GQp$ with determinant $\omega^t$ whose restriction to $I_p$ is $\omega_2^t\oplus\omega_2^{pt}$.
For $n\in\Z$ let $[n]$ denote the unique integer
in $\{0,1,2,\ldots,p-2\}$ congruent to~$n$ mod~$p-1$. We prove
the following result in this paper, via the technique explained above:

\begin{theorem}\label{thm:main}
If $k\geq 2$ and $0<v(a_p)<1$, and $t=[k-2]+1$, then either
$\overline{V}_{k,a_p}\cong\ind(\omega_2^t)$ is irreducible,
or $k\equiv3$~mod~$p-1$, and
$\overline{V}_{k,a_p}|_{I_p}\cong\omega\oplus\omega$.
\end{theorem}
Note that the theorem is true but vacuous if $p=2$. Note also that
for $p>2$, work of Breuil and Berger completely determines $V_{k,a_p}$
(for all $a_p$) in the case $k\leq 2p+1$
(see for example Th\'eor\`eme~3.2.1 of~\cite{ber05}).
Our contribution is hence
the case $k\geq 2p+2$ and $p>2$, where the assumption $v(a_p)<1$
of the theorem implies that the roots of $X^2-a_pX+p^{k-1}$ cannot
have ratio~1 or~$p$, which helps our exposition somewhat because
these correspond to ``special cases'' in the theory that sometimes
have to be dealt with separately.

If $k\not\equiv3$ modulo $p-1$ then the theorem tells you
$\overline{V}_{k,a_p}$, but if $k\equiv3$~mod~$p-1$ then there
are two possibilities in the conclusion of the theorem
and we know of no neat criterion to distinguish
between them. This initially surprising special
case actually has a simple global explanation.
Take a weight~3 newform of level $\Gamma_1(N)\cap\Gamma_0(p)$; it
will have slope $1/2$. Moreover the local mod $p$ representation
attached to this newform is ``hard'' to determine, because it
requires knowledge of the $\cal{L}$-invariant of the form.
Now any classical eigenform in a Coleman family sufficiently close
to this newform will be old at $p$, have weight congruent to~3
mod~$p-1$, and will still have slope $1/2$,
and hence the mod $p$ Galois representation attached to this eigenform
is also ``hard'' to determine and in particular will depend on
more than the slope of the form. On the positive side, if $p>2$
then we at least know what is happening for $k$ small: if $k=3$
or $k=p+2$ then the reducible case can never occur, and if $k=2p+1$
then a computation of Breuil explains exactly when the reducible
case occurs: see Th\'eor\`eme~3.2.1 of~\cite{ber05}.

\section{The mod $p$ and $p$-adic Langlands correspondences for $\GL_2(\Q_p)$.}

We recall some results on the mod~$p$ representations of $\GL_2(\Q_p)$.
Nothing in this section is due to the authors.
Let $p$ be a prime, let $\Zpbar$ be the integers in $\Qpbar$
and let $\Fpbar$ be the residue
field of $\Zpbar$. Say $r\in\Z_{\geq0}$.
Let $K$ be the group $\GL_2(\Z_p)$, and for
$R$ a $\Z_p$-algebra let $\Symm^r(R^2)$ denote the space
$\oplus_{i=0}^rRx^{r-i}y^i$ of homogeneous polynomials in two variables
$x$ and $y$, with the action of $K$ given by
$$\begin{pmatrix}a&b\\c&d\end{pmatrix}x^{r-i}y^i=(ax+cy)^{r-i}(bx+dy)^i.$$

Set $G=\GL_2(\Qp)$, and let $Z$ be its centre. If $V$ is an $R$-module
with an action of $K$, then extend
the action of $K$ to the group $KZ$ by letting $\smallmat{p}{0}{0}{p}$
act trivially, and let $I(V)$ denote the representation
$\ind_{KZ}^G(V)$ (compact induction).
Explicitly, $I(V)$ is the space of functions
$f:G\to V$ which have compact support modulo $Z$ and which
satisfy $f(kg)=k.(f(g))$ for all $k\in KZ$. This space has
a natural action of $G$, defined by $(gf)(\gamma)=f(\gamma g)$.
Note that \S2.2 of~\cite{barthel-livne} explains how an $R$-linear
$G$-endomorphism
of $I(V)$ can be interpreted as a certain function $G\to\End_R(V)$
(by Frobenius reciprocity). 

If $V=\Symm^r(R^2)$ for some integer $r\geq0$ and $\Z_p$-algebra $R$,
then, using
the dictionary mentioned above, Barthel and Livn\'e identified a certain
endomorphism $T$ of $I(V)$ which corresponds to the function
$G\to\End_R(V)$ which is supported on $KZ\smallmat{p}{0}{0}{1}KZ$ and
sends $\smallmat{p}{0}{0}{1}$ to the endomorphism of $\Symm^r(R^2)$
sending $F(x,y)$ to $F(px,y)$.
We refer to Lemme~2.1.4.1 of~\cite{breuil2}
and the remarks following this lemma
for many basic facts about this endomorphism.

We now recall the classification of smooth irreducible mod~$p$
representations of $\GL_2(\Q_p)$, due to Breuil and Barthel-Livn\'e.
For $r,n\in\Z_{\geq0}$ set $\sigma_r:=\Symm^r(\Fpbar^2)$ and set
$\sigma_r(n):=\det{}^{n}\otimes\Symm^r(\Fpbar^2)$. These are
$\Fpbar$-representations of~$K$, irreducible if $0\leq r\leq p-1$.
If $R$ is a ring and $t\in R^\times$ then let $\mu_t$ denote the
map $\GL_1(\Qp)\to R^\times$ which is trivial on $\Zp^\times$ and which
sends $p$ to $t$.
If $\chi:\Qp^\times\to\Fpbar^\times$ is a character, if
$\lambda\in\Fpbar$ and if $0\leq r\leq p-1$ then define
$$\pi(r,\lambda,\chi):=\left(I(\sigma_r)/(T-\lambda)\right)\otimes
(\chi\circ\det).$$ The classification is due to Barthel-Livne and
Breuil, and is as follows (see for example Th\'eor\`eme~2.7.1
of~\cite{breuil1}, which summarises results in~\cite{barthel-livne},
and also Corollaire~4.1.1, Corollaire~4.1.4 and Corollaire~4.1.5
of loc.\ cit.):

(i) If $\lambda\not=0$ and $(r,\lambda)\not\in\{(0,\pm1),(p-1,\pm1)\}$
then $\pi(r,\lambda,\chi)$ is irreducible.

(ii) There is a certain infinite-dimensional representation $\St$
of $G$, called the Steinberg representation. For $r\in\{0,p-1\}$
and $\lambda=\pm1$,
$\pi(r,\lambda,\chi)$ has two Jordan-Hoelder factors, one
1-dimensional and isomorphic to $\chi\mu_\lambda\circ\det$ and the
other equal to a twist of the Steinberg representation by this
same character. 

(iii) The representations $\pi(r,0,\chi)$ are all irreducible,
and are called the supersingular representations of $G$ (we remark
that this is the result due to Breuil and it is this part which
does not generalise to local fields other than $\Q_p$). No Jordan-Hoelder
factor of $\pi(r,\lambda,\chi)$ with $\lambda\not=0$ is supersingular.

(iv) We have just seen that all the representations $\pi(r,\lambda,\chi)$
have finite length. Conversely, any smooth irreducible $\Fpbar$-representation
of $G$ with a central character is a Jordan-Hoelder constituent
of some $\pi(r,\lambda,\chi)$.

(v) The only isomorphisms between the $\pi(r,\lambda,\chi)$
are the following:

(a) If $\lambda\not=0$ and  $(r,\lambda)\not\in\{(0,\pm1),(p-1,\pm1)\}$,
then
$\pi(r,\lambda,\chi)\cong\pi(r,-\lambda,\chi\mu_{-1})$,

(b) If $\lambda\not=0$ and $\lambda\not=\pm1$
then $\pi(0,\lambda,\chi)\cong\pi(p-1,\lambda,\chi)$
(and these are also isomorphic
to $\pi(0,-\lambda,\mu_{-1}\chi)$ and $\pi(p-1,-\lambda,\mu_{-1}\chi)$ as
already mentioned).

(c) $\pi(r,0,\chi)\cong\pi(r,0,\chi\mu_{-1})\cong\pi(p-1-r,0,\chi\omega^r)\cong
\pi(p-1-r,0,\chi\omega^r\mu_{-1})$.

Note that the Jordan-Hoelder factors
of $\pi(0,\lambda,\chi)$ and $\pi(p-1,\lambda,\chi)$ coincide
even if $\lambda=\pm1$.

We now move on to the $p$-adic part of the story.
Say $k\geq2$ and $a_p\in\Zpbar$ with $|a_p|<1$, as usual.
We now furthermore make the assumption that the roots of $X^2-a_pX+p^{k-1}$
do not have ratio~1 or~$p$; in other words we assume $a_p^2\not=4p^{k-1}$
and $a_p\not=\pm(1+p)p^{(k-2)/2}$. These assumptions are not always
necessary, but they make for a slightly cleaner exposition, and
in our final application we have $k\geq p+2$ and $v(a_p)<1$
and so they will be satisfied.

\begin{defn}Let $\Pi_{k,a_p}:=\ind_{KZ}^G\Symm^{k-2}(\Qpbar^2)/(T-a_p)$
(compact
induction, as before), and let  $\Theta_{k,a_p}$ be the image
of $\ind_{KZ}^G\Symm^{k-2}(\Zpbar^2)$ in
$\Pi_{k,a_p}$.

\end{defn}

Alternatively, $\Theta_{k,a_p}$ is
the quotient of $\ind_{KZ}^G\Symm^{k-2}(\Zpbar^2)/(T-a_p)$ by its torsion.
Then $\Pi_{k,a_p}$ is irreducible
by Proposition~3.3(i) of~\cite{breuil2} and $\Theta_{k,a_p}$ is
a lattice in it, by Corollaire 5.3.4 of \cite{berger-breuil}.

The $p$-adic Langlands correspondence associates a
unitary $p$-adic Banach space representation $B(V)$ of
$\GL_{2}(\Q_{p})$ to an irreducible crystalline representation $V$,
and $B(V_{k,a_p})$ (under our assumptions on $k$ and $a_p$ above)
is isomorphic to the completion of $\Pi_{k,a_{p}}$ with
respect to the gauge
of $\Theta_{k,a_{p}}$. We deduce that $\Bbar(V_{k,a_p})$
(the semisimplification of the reduction of an open $\GL_2(\Q_p)$-stable
lattice in $B(V_{k,a_p})$) is
isomorphic to $\Thetabar_{k,a_p}^{\sss}$, where
$\Thetabar_{k,a_p}:=\Theta_{k,a_p}\otimes\Fpbar$. We now recall
the explicit mod $p$ local Langlands correspondence for $\GL_2(\Q_p)$
formulated by Breuil and the compatibility of the mod~$p$
and $p$-adic Langlands correspondences for trianguline representations
proved by Berger (see
Th\'eor\`eme~A of~\cite{ber05}), and in particular what these things
tell us about the situation in hand.
Let $[x]$ be
the unique integer in $[0,p-2]$ congruent to $x\in\Z$
modulo $p-1$.

\begin{theorem}\label{thm:bergermodp}

We have
$$\overline{V}\cong  (\ind(\omega_{2}^{r+1}))\otimes \chi
\iff
\Bbar(V)\cong \pi(r,0,\chi) $$
and
$$\overline{V}\cong  (\mu_{\lambda}\omega^{r+1}\oplus \mu_{\lambda^{-1}})\otimes \chi
\iff
\Bbar(V)\cong \pi(r,\lambda,\chi)^{ss}\oplus\pi([p-3-r],\lambda^{-1},\chi\omega^{r+1})^{ss}.$$
\end{theorem}

In particular, if $k\geq2$, $a_p\in\Qpbar$ with $v(a_p)>0$,
and if the roots of $X^2-a_pX+p^{k-1}$ don't have ratio $1$ or $p^{\pm1}$,
we have
$$\overline{V}_{k,a_{p}}\cong  (\ind(\omega_{2}^{r+1}))\otimes \chi
\iff
(\Thetabar_{k,a_p})^{\sss}\cong \pi(r,0,\chi) $$
and
$$\overline{V}_{k,a_{p}}\cong  (\mu_{\lambda}\omega^{r+1}\oplus \mu_{\lambda^{-1}})\otimes \chi
\iff
(\Thetabar_{k,a_p})^{\sss}\cong \pi(r,\lambda,\chi)^{\sss}\oplus\pi([p-3-r],\lambda^{-1},\chi\omega^{r+1})^{\sss}.$$

\section{Lemmas about mod~$p$ representations of $\GL_2(\Q_p)$.}

Our strategy for proving Theorem~\ref{thm:main}
is to compute $\Vbar_{k,a_p}$ when $0<v(a_p)<1$ by
analysing $\Thetabar_{k,a_p}$ and its possible Jordan-Hoelder factors.
The following lemma follows
directly from the
explicit description of the Jordan-Hoelder factors of $\pi(r,\lambda,\chi)$
given in the previous section.
\begin{lemma}\label{lem:modpequiv} If $\lambda\not=0$
and $\pi(r,\lambda,\chi)$
and $\pi(r',\lambda',\chi')$ have a common Jordan-Hoelder factor, then $\lambda'\not=0$, $r\equiv r'$ mod $p-1$,
and $\chi/\chi'$ is unramified.
\end{lemma}\hfill$\square$

The next lemma is a straightforward strengthening of Proposition~32
of~\cite{barthel-livne}.

\begin{lemma}\label{3.2} If $0\leq r\leq p-1$
and $F$ is a quotient of $I(\sigma_r)$ which has finite
length as an $\Fpbar[G]$-module, then every Jordan-Hoelder factor
of $F$ is a subquotient of
$I(\sigma_r)/(T-\lambda)=\pi(r,\lambda,1)$ for some $\lambda$ (with $\lambda$
possibly depending on the factor).
\end{lemma}
\begin{proof} Induction on the length of $F$. If $F$ is irreducible
then Th\'eor\`eme~2.7.1(i) of~\cite{breuil1} shows that in fact $F$
is a quotient of some $\pi(r,\lambda,1)$, and in particular a subquotient.
Assume now that $F$ has
length greater than one, so that $F$ has
an irreducible quotient $J$, and the kernel $K$ has smaller length.
Again by loc.\ cit., the composite map $I(\sigma_r)\to F\to J$ factors as
$I(\sigma_r)\to\pi(r,\lambda,1)\to J$ for some $\lambda$. Let $\gamma$
denote this latter arrow.

Now consider the following commutative diagram:

$$\xymatrix{& I(\sigma_r)\ar[r]^{T-\lambda}\ar[d]^{\alpha}& I(\sigma_r)\ar[r]\ar[d]^{\beta}& \pi(r,\lambda,1)\ar[r]\ar[d]^{\gamma}& 0  \\ 
0\ar[r] & K\ar[r]& F\ar[r]& J\ar[r]& 0 }$$

Since $\beta$ is surjective, the snake lemma implies that $\coker(\alpha)$ is
a subquotient of $\pi(r,\lambda,1)$ and hence
each of its Jordan-Hoelder factors are too.
Moreover, the inductive hypothesis
tells us that every Jordan-Hoelder factor of $\im(\alpha)$ is a subquotient
of $\pi(r,\lambda',1)$ for some $\lambda'$ (depending on the factor),
hence any Jordan-Hoelder factor of $K$ and hence of $F$ is a subquotient
of $\pi(r,\lambda',1)$ for some $\lambda'$ and we are done.
\end{proof}

The following proposition will be crucial for us. It uses the
previous lemma to conclude
that we can say a lot about a finite length
$\Fpbar$-representation of $\GL_2(\Q_p)$ which is
a quotient of $I(W)$ for some \emph{irreducible} $W$,
and whose semisimplification is in the image of $LL$. Note
that any supersingular representation of $\GL_2(\Q_p)$ is
a quotient of $I(W)$ for some irreducible $W$; the point
is that almost no other elements of the image of $LL$ are.
Recall that if $W$ is an irreducible $\Fpbar$-representation
of $K$ then $W=\sigma_s(n)$ for some $0\leq s\leq p-1$.

\begin{proposition}\label{prop:scalaroninertia}

If $V$ is an irreducible crystalline
representation of $\GQp$, if $\Theta$ is an open $\GL_2(\Q_p)$-stable
lattice in $B(V)$ and if there is a surjection
$I(\sigma_s(n))\to\Thetabar$ for some $s$ with $0\leq s\leq p-1$,
then either $\Vbar\cong\ind(\omega_2^{s+1+(p+1)n})$ is irreducible,
or $\Vbar$ is reducible and scalar on inertia. If $p>2$
then in the reducible case we must furthermore have $s=p-2$ and
$\Vbar|_{I_p}=\omega^n\oplus\omega^n$.
\end{proposition}
\begin{remark} If $p=2$ then this proposition is vacuous, as if $p=2$
then \emph{every}
semisimple 2-dimensional mod $p$ representation of $\Gal(\Qpbar/\Qp)$
is either irreducible or scalar on inertia.
\end{remark}
\begin{proof} Firstly observe that
$I(\sigma_s(n))\cong I(\sigma_s)\otimes(\omega\circ\det)^n$ as
$\GL_2(\Q_p)$-representations, so (twisting~$V$ by an appropriate
power of the cyclotomic character) we may assume that $n=0$
(as $LL$ is compatible with twists).

We know that $\Thetabar^{\sss}$ must
be in the image of $LL$. Hence if $\Thetabar^{\sss}$ is irreducible
then it is of the form $\pi(r,0,\chi)$ for some $r,\chi$. However
Lemma~\ref{3.2} tells us that
any irreducible quotient of $I(\sigma_s)$ must be a Jordan-Hoelder
factor of $\pi(s,\lambda,1)$ for some $\lambda$, and by the classification
theorem we must have $\pi(r,0,\chi)\cong\pi(s,0,1)$.
By Theorem~\ref{thm:bergermodp}
we deduce $\Vbar\cong\ind(\omega_2^{s+1})$ in this case.

So now say $\Thetabar^{\sss}$ is reducible (and hence $\Vbar$ is
reducible). The crucial point is
that if $\lambda\in\Fpbar^\times$ then usually the Jordan-Hoelder factors
of $\pi(r,\lambda,\chi)^{\sss}
\oplus\pi([p-3-r],\lambda^{-1},\chi\omega^{r+1})^{\sss}$
(a general reducible element of the image of $LL$)
cannot \emph{all} be subquotients of $I(\sigma_s)$ for the same $s$.
Indeed, applying Lemma~\ref{3.2} to $\Thetabar$
we see that if $\Vbar$ is reducible then
the Jordan-Hoelder factors of $\pi(r,\lambda,\chi)$
and $\pi([p-3-r],\lambda^{-1},\chi\omega^{s+1})$ \emph{all}
have to be subquotients
of the $\pi(s,\lambda',1)$ for varying $\lambda'$, and
by Lemma~\ref{lem:modpequiv} above we note
that this forces both $\chi$ and $\chi\omega^{r+1}$ to be unramified,
so $r\equiv s\equiv p-2$ modulo~$p-1$, and in this case we see from
Theorem~\ref{thm:bergermodp} that $\Vbar$ is unramified.
\end{proof}

\section{The kernel of the map $I(\sigma_{k-2})\to\Thetabar_{k,a_p}$.}

Let us re-iterate our assumptions: $k\geq2$, $v(a_p)>0$ and the roots of
$X^2-a_pX+p^{k-1}$ do not have ratio $1$ or $p^{\pm1}$.
It is clear from the definition that
$\Thetabar_{k,a_p}$ admits a surjection from $I(\sigma_{k-2})/(T)$
and in particular from $I(\sigma_{k-2})$. Note however that
Proposition~\ref{prop:scalaroninertia} does not apply to this
situation, because $\sigma_{k-2}$ is not in general irreducible.

Let $X(k,a_p)$ denote the kernel of the surjection
$I(\sigma_{k-2})\to\Thetabar_{k,a_p}$. In this section we will
analyse $X(k,a_{p})$, and as a consequence
find, in certain cases, that $\Thetabar_{k,a_p}$ does admit a surjection
from $I(W)$ for some \emph{irreducible} $\Fpbar$-representation~$W$ of $K$,
enabling us to apply Proposition~\ref{prop:scalaroninertia}. 

Firstly, we establish some notation, following \cite{breuil2}.
Recall that for $V$ a $\Z_p[K]$-module, the space $I(V)$
was defined previously to be a certain space of functions $G\to V$.
We let $[g,v]$ denote the (unique) element of $I(V)$
which is supported on $KZg^{-1}$,
and which satisfies $[g,v](g^{-1})=v$. Note that $g[h,v]=[gh,v]$ for
$g,h\in G$, that $[gk,v]=[g,kv]$ for $k\in K$, and that
the $[g,v]$ span $I(V)$ as an abelian group, as $g$ and $v$ vary.
For $\lambda\in\Q_p$
write $g^0_{1,\lambda}:=\smallmat{p}{\lambda}{0}{1}$, and
set $\alpha:=\smallmat{1}{0}{0}{p}$. Now let $V=\Symm^{k-2}(R^2)$ for
some $\Z_p$-algebra $R$, thought of, as usual,
as homogeneous polynomials in the variables $x$ and $y$
of degree~$k-2$. An easy consequence of Lemma 2.2.1 of \cite{breuil2}
is that $T([\Id,v])=T^+([\Id,v])+T^-([\Id,v])$
where
$$T^+([\Id,v])=\sum_{\lambda\in\Z_p:\lambda^p=\lambda}
[g^0_{1,\lambda},v(x,py-\lambda x)]$$
and
$$T^-([\Id,v])=[\alpha,v(px,y)].$$

For simplicity now, write $r:=k-2\in\Z_{\geq0}$. Note that $I$
is an exact functor, so if $W\subseteq V$ are $K$-representations
then $I(W)$ is naturally a subset of $I(V)$.

\begin{lemma}\label{lem:imageofT} If $r\geq p$ then
$T.I(\sigma_r)=I(W_r)$
where $W_r$ is the $\Fpbar[K]$-submodule of $\sigma_r$
generated by $y^r$.
\end{lemma}
\begin{remark} If $r\leq p-1$ then the image of $T$ is harder
to describe. It is a minor miracle that one has such a succinct
description of the image of $T$ when $r\geq p$.
\end{remark}
\begin{proof} It is clear that
$T.I(\sigma_r)$ is generated
as an $\Fpbar[G]$-module by elements of the form $T([\Id,v])$, for
$v\in\sigma_r$. By our explicit formula for $T([\Id,v])$ above (which
simplifies because $p=0$ in this situation), we
see that $T([\Id,v])$ may be written as a sum of terms $[h,w]$,
all of which have the property that the $w$s are in
$\Fpbar x^r$ or $\Fpbar y^r$. Because $\smallmat{0}{1}{1}{0}y^r=x^r$,
we see that $T.I(\sigma_r)$ is certainly
contained in $I(W_r)$.

For the converse inclusion, define
$$\overline{\theta}:=xy^p-yx^p
=x\prod_{\lambda:\lambda^{p}=\lambda}(y-\lambda x)\in\Fpbar[x,y].$$
Then if $f:=[\Id,y^{r-p}\overline{\theta}/x]\in I(\sigma_r)$
we have $T^+f=0$, and so $Tf=T^-f=[\alpha,y^r]$, which is easily
checked to generate $I(W_r)$ as an $\Fpbar[G]$-module.
\end{proof}

Recall that $X(k,a_p)$ is the kernel of the map
$I(\sigma_r)\to\Thetabar_{k,a_p}$.
Multiplication by $\thetabar$ induces a map $\sigma_r\to\sigma_{r+p+1}$,
which is $K$-equivariant when thought of as a map
$\sigma_r(1)\to\sigma_{r+p+1}$.

\begin{lemma}\label{lem:smallvinfo} If $v(a_p)<1$ and
$r\geq p+1$ then $I(\thetabar\sigma_{r-(p+1)})\subseteq X(k,a_p).$
\end{lemma}

\begin{proof} Set $\theta=x^{p}y-xy^{p}\in\Qpbar[x,y]$.
For $g\in\Zpbar[x,y]$ of degree $r-(p+1)$, set
$f=[\Id,a_p^{-1}\theta g]\in\ind_{KZ}^G(\Symm^r(\Qpbar^2)).$
Clearly the image of $Tf-a_pf$ in $\Pi_{k,a_p}$ is zero.
On the other hand, $\theta(x,py-\lambda x)$ (for $\lambda\in\Z_p$)
and $\theta(px,y)$
are in $p\Z_p[x,y]$, and so by the explicit formula for $T$ above
one deduces that $Tf\in(p/a_p)\ind_{KZ}^G\Symm^r(\Zpbar^2))$.
In particular $Tf$
is zero in $\Thetabar_{k,a_p}$. Yet $a_pf$
is $[\Id,\theta g]$, so $Tf-a_pf$ is in $\ind_{KZ}^G\Symm^r(\Zpbar^2))$,
and $[\Id,\thetabar\overline{g}]$ is in $X(k,a_p)$.
\end{proof}

\begin{remark} Although we will not use these results here (but they
could be used if one were trying to formulate an analogue of
Theorem~\ref{thm:main} for larger $v(a_p)$), we note
that similar tricks show that if $v(a_p)<t\in\Z_{\geq1}$ and
$r\geq t(p+1)$ then $I(\thetabar^t\sigma_{r-t(p+1)})\subseteq X(k,a_p)$
(set $f=[\Id,a_p^{-1}\theta^tg]$) and that if $v(a_p)>n\in\Z_{\geq0}$,
if $0\leq i\leq n$ and $r\geq i(p+1)+p$ then
$I(\langle x^iy^{r-i}\rangle_K)\subseteq X(k,a_p),$
where $\langle x^iy^{r-i}\rangle_K$ denotes the sub-$K$-representation
of $\sigma_r$ generated by $x^iy^{r-i}$
(set $f=[\Id,(\theta/p)^i(\theta/x)y^{r-i(p+1)-p}]$).
\end{remark}

\section{Analysis of $0<v<1$.}
If $k\geq p+3$ and $0<v(a_p)<1$ then the roots of $X^2-a_pX+p^{k-1}$
cannot have ratio $1$ or $p^{\pm1}$ so the results of the preceding
sections give
\begin{corollary}\label{cor:ashstevens} If $0<v(a_{p})<1$ and $k\geq p+3$
then $X(k,a_p)$
contains $\ind_{KZ}^G(Y)$ where $Y$ is the sub-$K$-representation
of $\sigma_r$ generated by $\thetabar\sigma_{r-(p+1)}$
and $y^r$. In particular the map $\ind_{KZ}^G(\sigma_r)\to\Thetabar_{k,a_p}$
factors through the induction $\ind_{KZ}^G(\sigma_r/Y)$
of the \emph{irreducible}
representation $\sigma_r/Y$. We have $\sigma_r/Y\cong\sigma_s(r)$
where $0\leq s\leq p-2$ and $s\equiv p-1-r$ mod~$p-1$.\end{corollary}
\begin{proof} The first part follows from Lemmas \ref{lem:imageofT} and
\ref{lem:smallvinfo}. The rest of the Corollary follows once one knows the
isomorphism $\sigma_r/Y\cong\sigma_s(r)$, which follows
from Lemma~3.2(a) and~(c) of~\cite{ashstevens}.
\end{proof}
\begin{corollary} Theorem~1.6 is true.
\end{corollary}
\begin{proof} By the comments following the statement of the theorem,
we need only deal with the case $p>2$ and $k\geq 2p+2$. In fact
we prove the result for $p>2$ and $k\geq p+3$. First fix~$s$
such that $0\leq s\leq p-2$ and $s\equiv 2-k$~mod~$p-1$.
By Corollary~\ref{cor:ashstevens} we see $\Thetabar_{k,a_p}$
is a quotient of $I(\sigma_s(k-2))$. 
The result now follows from Proposition \ref{prop:scalaroninertia}
and some elementary arithmetic.
\end{proof}

\begin{corollary} If $f$ is a modular form of weight $k\geq p+3$,
level prime to $p$, and $0<v(a_p)<1$, and if $0\leq g\leq p-2$
with $g\equiv k-2$~mod~$p-1$, then either
$(\rhobar_{f}|_{D_{p}})^{ss}\cong\ind(\omega_2^{g+1})$,
or $k\equiv3$ mod $p-1$ and
$(\rhobar_{f}|_{I_{p}})^{ss}\cong\omega\oplus\omega$.
\end{corollary}
\begin{proof}
This follows from the preceding corollary and Theorem~\ref{1.4}.
\end{proof}
\begin{remark} For $k\leq p+2$ slightly stronger results are known
(for example
if $k\leq p+2$ and $p>2$ then the results of \cite{breuil2} show that
the reducible case can't occur). On the other hand when $k=2p+1$
and $p>2$, unpublished calculations of Breuil, and numerical examples due to
one of us (KB) show that when $v(a_p)=1/2$ both reducible and irreducible
possibilities can occur.

\end{remark}

\begin{corollary}\label{cor:eigencurve}If $2\leq p\leq 53$, then the slopes
of level $1$ modular forms of all weights are never in the range $(0,1)$.
\end{corollary}
\begin{proof}If $p=2$ this follows from the results of \cite{MR516060}.
If $p$ is odd then  the congruence $k=3$ mod $p-1$ implies
that $k$ is odd, and there are no level~1 forms of odd weight.
The corollary would follow if we knew that for $f$ a modular form
of level~1 and $p\leq 53$ then $\rhobar_{f}|_{D_{p}}$ is always reducible.
But Corollary~3.6 of~\cite{ashstevens} reduces this statement to
a finite check (we only need to check level~1 eigenforms of weight at
most $p+1\leq 54$), and by~Theorem~\ref{thm:deligneserrefontaineedixhoven}
we see that what we must verify is that if $2<p\leq 53$ then the
$T_p$-eigenvalues on the space of cusp forms of level~1
and weight~$k\leq p+1$ are all $p$-adic units, which is easily checked
nowadays on a computer.
\end{proof}
\begin{remark} If $p=59$ and $f$ is the normalised level~1 weight~16
cusp form then $f$ is non-ordinary at~$p$, the local mod~$p$ representation
is irreducible, and there do exist level~1 modular forms with slope
equal to $1/2$ (for example, in weight~74).
\end{remark}

\providecommand{\bysame}{\leavevmode\hbox to3em{\hrulefill}\thinspace}
\providecommand{\MR}{\relax\ifhmode\unskip\space\fi MR }

\providecommand{\MRhref}[2]{

  \href{http://www.ams.org/mathscinet-getitem?mr=#1}{#2}
}
\providecommand{\href}[2]{#2}

\end{document}